\documentclass[12pt]{article}
\newcommand{\keywords}[1]{\textbf{Keywords:}\quad #1}
\newcommand{\AMS}[1]{\textbf{AMS-classification numbers:}\quad #1}
\usepackage{geometry}                
\usepackage{graphicx}
\usepackage{amssymb}
\usepackage{epstopdf}
\usepackage{amsthm}
\usepackage{amsmath}
\usepackage[small]{complexity}
\usepackage{tikz}
\usepackage{anysize}
\usepackage{lscape}
\usepackage[latin1]{inputenc}
\usepackage{ytableau }
\usepackage{youngtab}
\newtheorem{thm}{Theorem}
\newtheorem{prop}{Proposition}
\newtheorem{cor}{Corollary}
\newtheorem{lemma}{Lemma}
\newtheorem{definition}{Definition}

\title{A Fast Fourier Transform for the Johnson graph}

\author{Rodrigo Iglesias\\
\small Departamento de Matemática\\ [-0.8ex]
\small Universidad Nacional del Sur\\ [-0.8ex]
\small Bahía Blanca, Argentina.
 \and
Mauro Natale\\
\small Facultad de Ciencias Exactas\\ [-0.8ex]
\small Universidad Nacional del Centro\\ [-0.8ex]
\small Tandil, Argentina.} 

\begin{document}
\maketitle

\begin{abstract}
The set $X$ of $k$-subsets of an $n$-set has a natural graph structure where two $k$-subsets are connected if and only if the size of their intersection is $k-1$. This is known as the Johnson graph.
The symmetric group $S_n$ acts on the space of complex functions on $X$ and this space has a multiplicity-free decomposition as sum of irreducible representations of $S_n$, so it has a well-defined Gelfand-Tsetlin basis up to scalars. The Fourier transform on the Johnson graph is defined as the change of basis matrix from the delta function basis to the Gelfand-Tsetlin basis.

The direct application of this matrix to a generic vector requires $\binom{n}{k}^2$ arithmetic operations. 
We show that --in analogy with the classical Fast Fourier Transform on the discrete circle-- this matrix can be factorized as a product of $n-1$ orthogonal matrices, each one with at most two nonzero elements in each column.  The factorization is based on the construction of $n-1$ intermediate bases which are parametrized via the Robinson-Schensted insertion algorithm.
This factorization shows that the number of arithmetic operations required to apply this matrix to a generic vector is bounded above by $2(n-1) \binom{n}{k}$. 

We show that each one of these sparse matrices can be constructed using $O(\binom{n}{k})$ arithmetic operations. Our construction does not depend on numerical methods. Instead, they are obtained by solving small linear systems with integer coefficients derived from the Jucys-Murphy operators. Then both the construction and the succesive application of all these  $n-1$ matrices can be performed using $O(n  \binom{n}{k})$ operations.

As a consequence, we show that the problem of computing all the weights of the isotypic components of a given function can be solved in $O(n \binom{n}{k})$ operations, improving the previous bound $O(k^2 \binom{n}{k})$ when $k$ asymptotically dominates $\sqrt{n}$. The same improvement is achieved for the problem of computing the isotypic projection onto a single component.


\end{abstract}

 
 \AMS{65T50, 43A30, 68W30, 68R05}

\section{Introduction}
The set of all subsets of cardinality $k$ of a set of cardinality $n$ is a basic combinatorial object with a natural metric space structure where two $k$-subsets are at distance $d$ if the size of their intersection is $k-d$. This structure is captured by the Johnson graph $J(n,k)$, whose nodes are the $k$-subsets and two $k$-subsets are connected if and only if they are at distance $1$. 

The Johnson graph is closely related to the Johnson scheme, an association scheme of major significance in classical coding theory (see \cite{Delsarte} for a survey on association scheme theory and its application to coding theory). 
The Johnson graph has played a fundamental role in the breakthrough quasipolynomial time algorithm for the  graph isomorphism problem presented in \cite{Babai} (see \cite{Toran} for background on the graph isomorphism problem).

  Functions on the Johnson graph arise in the analysis of ranked data. In many contexts, agents choose a $k$-subset from an $n$-set, and the data is collected as the function that assigns to the $k$-subset  $x$ the number of agents who choose $x$. This situation is considered, for example, in the statistical analysis of certain lotteries (see \cite{Diaconis2}, \cite{Diaconis-Rockmore}).
  
The vector space of functions on the Johnson graph is a representation of the symmetric group and it  decomposes as a multiplicity-free direct sum of irreducible representations (see \cite{Stanton}). 
Statistically relevant information about the function is contained in the isotypic projections of the function onto each irreducible component. This approach to the analysis of ranked data was called spectral analysis by Diaconis and developed in \cite{Diaconis}, \cite{Diaconis2}.
The problem of the efficient computation of the isotypic projections  has been studied by Diaconis and Rockmore in \cite{Diaconis-Rockmore}, and by Maslen, Orrison and Rockmore in \cite{Lanczos}. 
  
  The classical Discrete Fourier Transform (DFT) on the cyclic group  $\mathbb{Z}/2^n \mathbb{Z}$ can be seen as the application of a change of basis matrix from the  basis  $B_0$ of delta functions  to the basis $B_n$ of characters of the group $\mathbb{Z}/2^n \mathbb{Z}$. The direct application of this matrix to a generic vector involves $(2^n)^2$ arithmetic operations. The Fast Fourier Transform (FFT) is a fundamental algorithm that computes the DFT in $O(n 2^n )$ operations. This  algorithm was discovered by Cooley and Tukey \cite{Cooley-Tukey} and  the efficiency of their algorithm is due to a factorization of the change of basis matrix
  $$
  [B_0]_{B_n} = [B_{n-1}]_{B_n}\  ... \ [B_1]_{B_2} \ [B_0]_{B_1}
  $$
  where $B_1 ,..., B_{n-1}$ are intermediate orthonormal bases such that each matrix $[B_{i-1}]_{B_i}$ has  at most two nonzero entries in each column. We denote by $[B]_{B'}$ the change of basis matrix from the base $B$ to the base $B'$.
  
  In this paper, we show that the same phenomenon occurs in the case of the non-abelian Fourier transform on the Johnson graph.  This transform is defined as the application  of the change of basis matrix from the  basis  $B_0$ of delta functions  to the basis $B_n$ of Gelfand-Tsetlin functions. The Gelfand-Tsetlin basis --defined in Section \ref{GT basis}--  is well-behaved with respect to the action of the symmetric group $S_n$, in the  sense that each irreducible component is generated by a subset of the basis.

 The computational model used here counts a single complex multiplication
and addition as one operation.
We remark that we only count these algebraic operations and do not count those operations involved in the storage of the matrices. For example, we do not count the operations needed to reorder the rows and columns of a matrix.

 A direct computation of this Fourier transform involves $\binom{n}{k}^2$ arithmetic operations.
  We construct intermediate orthonormal bases $B_1 ,..., B_{n-1}$ such that each change of basis matrix $[B_{i-1}]_{B_i}$ has at most two nonzero entries in each column. 
  Each intermediate basis $B_i$ is parametrized by pairs composed by a standard Young tableau of height at most two and a word in the alphabet $\{ 1,2 \}$ as shown in Figure \ref{Labels of intermediate bases}.
  These intermediate bases enable the computation of the non-abelian Fourier transform --as well as its inverse-- in at most $2(n-1) \binom{n}{k}$ operations. 
  Our construction of the matrices $[B_{i-1}]_{B_i}$ is based on the Vershik-Okounkov approach \cite{Vershik} to the representation theory of the symmetric groups, which uses the Jucys-Murphy operators as a basic tool. 
  The present paper is an extension of \cite{Iglesias-Natale}. This previous work not included the efficient construction of the matrices $[B_{i-1}]_{B_i}$.  
  
  \begin{figure}[ht] 
\label{Labels of intermediate bases}
\begin{center}
\includegraphics[scale=0.6]{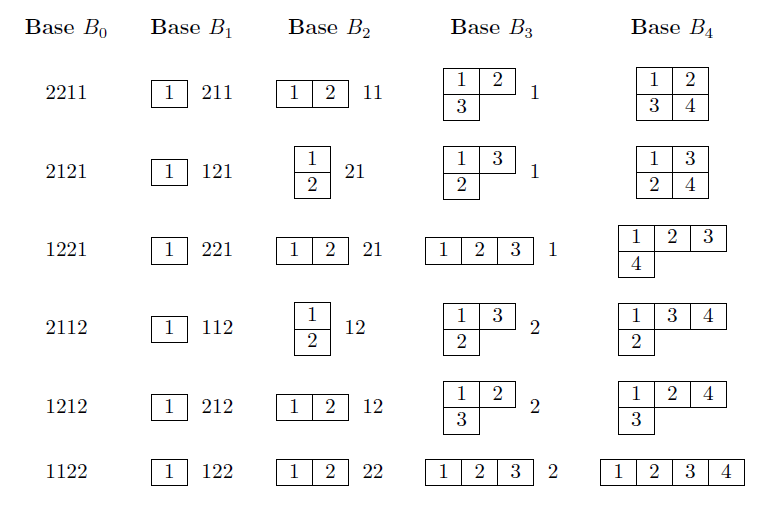} 
\caption{Labels of the intermediate bases in the case $n=4, k=2$. The $i$-th column parametrize the basis $B_i$. }
\label{Labels of intermediate bases}
\end{center}
\end{figure}

In \cite{Wolff}, Gelfand-Tsetlin bases were defined in the contextt of semisimple algebras and fast Fourier transforms were given for BMW, Brauer, and Temperley-Lieb algebras.  The function space on the Johnson graph is not a semisimple algebra, but it can be viewed as a module over the group ring $\mathbb{C}[S_n]$.  Our work is perhaps an indication that the methods in \cite{Wolff} extend to interesting modules over semisimple algebras.

  The upper bound we obtained for the algebraic complexity of the Fourier transform on the Johnson graph can be applied to the well-studied problem of computing the isotypic components of a function.
The most efficient algorithm for computing all the isotypic components --given by Maslen, Orrison and Rockmore in \cite{Lanczos}--  depends on Lanczos iteration method and uses $O(k^2 \binom{n}{k})$ operations. 
If the problem were to compute the isotypic projection onto a single component, it is no clear how to reduce this upper bound using the algorithm in \cite{Lanczos}. We show that -once the intermediate matrices $[B_{i-1}]_{B_i}$ have been computed for a fixed pair $(n,k)$-- this task can be accomplished in $O(n \binom{n}{k})$ operations, so our upper bound is an improvement when $k$ asymptotically dominates $\sqrt{n}$. We remark that our method does not depend on numerical computations. Instead, our construction of the matrices that perform the fast Fourier transform is based on the application of exact arithmetic operations given by the Jucys-Murphy operators.

We also show that the same $O(n \binom{n}{k})$ bound is achieved for the problem of computing all the weights of the isotypic components appearing in the decomposition of a function.  This problem could also be solved by computing every isotypic component and measuring their lengths, but this approach requires $O(k^2 \binom{n}{k})$ operations if we use the algorithm in \cite{Lanczos}.

In Section \ref{GT basis}, we review the definition  of Gelfand-Tsetlin bases for representations of the symmetric group.
In Section  \ref{Johnson graph}, we describe the well-known decomposition of the function space on the Johnson graph and define the corresponding Gelfand-Tsetlin basis.
In Section \ref{Decomposition theorems} we prove some decomposition theorems for the function space --Theorems \ref{One-dimensional decomposition} and \ref{Adapted decomposition}-- which are central for our subsequent results.
In Section \ref{Intermediate}, we introduce the sequence of intermediate bases of the function space, we prove that the change of basis matrix between two consecutive bases is a sparse matrix (Theorem \ref{Sparsity of matrix}) and we give a upper bound for number of operations used to apply the Fourier transform to a vector (Theorem \ref{Main theorem}). 
In Section \ref{RS} we point out the relation of our algorithm with the Robinson-Schensted insertion algorithm.
In Section \ref{Jucys-Murphy operators} we present some basic tools from the Vershik-Okounkov approach to the representation theory of the symmetric group, namely, the Jucys-Murphy operators and the formula for their eigenvalues.
In Section \ref{Efficient computation of the matrices} we give an efficient construction of the sparse matrices that realize the fast Fourier transform based on the properties of the Jucys-Murphy operators, and prove our main result (Theorem \ref{Main}).
In Section \ref{Applications}  we apply our algorithm to the problem of the computation of the isotypic components of a function on the Johnson graph, obtaining an improvement when $k$ asymptotically dominates $\sqrt{n}$.

\section{Gelfand-Tsetlin bases}
\label{GT basis}

Consider the chain of subgroups of $S_n$ 
$$
S_1 \subset S_2 \subset S_3 ...\subset S_n
$$
where $S_k$ is the subgroup of those permutations fixing the last $n-k$ elements of $\{1,...,n\}$.
Let $Irr(n)$ be the set of equivalence classes of irreducible complex representations of $S_n$. 
A fundamental fact in the representation theory of $S_n$ is that if $V_{\lambda}$ is an irreducible $S_n$-module corresponding to the representation $\lambda \in Irr(n)$ and we consider it by restriction as an $S_{n-1}$-module, then it decomposes as sum of irreducible representations of $S_{n-1}$ in a multiplicity-free way (see for example \cite{Vershik}). This means that if $V_{\mu}$ is an irreducible $S_{n-1}$-module corresponding to the representation $\mu \in Irr(n-1)$ then the dimension of the space 
$
\mbox{Hom}_{S_{n-1}} (V_{\mu},V_{\lambda})
$
is $0$ or $1$.
The \textit{branching graph}  is the following directed graph. The set of nodes is the disjoint union 
$$
\bigsqcup_{n \ge 1}Irr(n).
$$
Given representations $\lambda \in Irr(n)$ and $\mu \in Irr(n-1)$ there is an edge connecting them if and only if  
$\mu$ appears in the decomposition of $\lambda$, that is, if $\mbox{dim Hom}_{S_{n-1}} (V_{\mu},V_{\lambda})=1$.
If there is an edge between them we write
$$
\mu \nearrow \lambda,
$$
so we have a canonical decomposition of $V_{\lambda}$ into irreducible $S_{n-1}$-modules
$$
V_{\lambda} = \bigoplus_{\mu \nearrow \lambda} V_{\mu} .
$$
Applying this formula iteratively we obtain a uniquely determined decomposition into one-dimensional subspaces 
$$
V_{\lambda} = \bigoplus_{T} V_{T} ,
$$
where $T$ runs over all chains
$$
T = \lambda_1 \nearrow \lambda_2 \nearrow ... \nearrow \lambda_n,
$$
with $\lambda_i \in Irr(i)$ and $\lambda_n = \lambda$.
Choosing a unit vector $v_T$ --with respect to the $S_n$-invariant inner product  in $V_\lambda$--  of the one-dimensional space $V_T$ we obtain a basis $\{v_T\}$ of the irreducible module $V_{\lambda}$, which is called the \textit{Gelfand-Tsetlin basis}. 

Observe that if $V$ is a multiplicity-free representation of $S_n$ then there is a uniquely determined --up to scalars--  Gelfand-Tsetlin basis of $V$. In effect, if 
$$
V = \bigoplus_{\lambda \in S \subseteq Irr(n)} V_{\lambda}
$$
and  $B_{\lambda}$ is a GT-basis of $V_{\lambda}$ then a GT-basis of $V$ is given by the disjoint union 
$$
B = \bigsqcup_{\lambda \in S \subseteq Irr(n)} B_{\lambda}.
$$

The Young graph is the directed graph where the nodes are the Young diagrams and there is an arrow from  $\lambda$ to  $\mu$ if and only if $\lambda$ is contained in $\mu$ and their difference consists in only one box. It turns out that there is a bijection between the set of Young diagrams with $n$ boxes and $Irr(n)$ inducing a graph isomorphism between the Young graph and the branching graph (Theorem 5.8 of \cite{Vershik}). 
  
Then there is a bijection between the Gelfand-Tsetlin basis of $V_{\lambda}$ --where $\lambda$ is a Young diagram-- and the set of paths in the Young graph starting at  the one-box diagram and ending at the diagram $\lambda$. Each path can be represented by a unique standard Young tableau, so that the Gelfand-Tsetlin-basis of $V_{\lambda}$ is parametrized by the set of standard Young tableaux of shape $\lambda$ (see Figure \ref{Young graph}).
From now on we identify a chain 
$\lambda_1 \nearrow \lambda_2 \nearrow ... \nearrow \lambda_n$
with its corresponding standard Young tableau.

\begin{figure}[ht] 
\begin{center}
\includegraphics[scale=0.4]{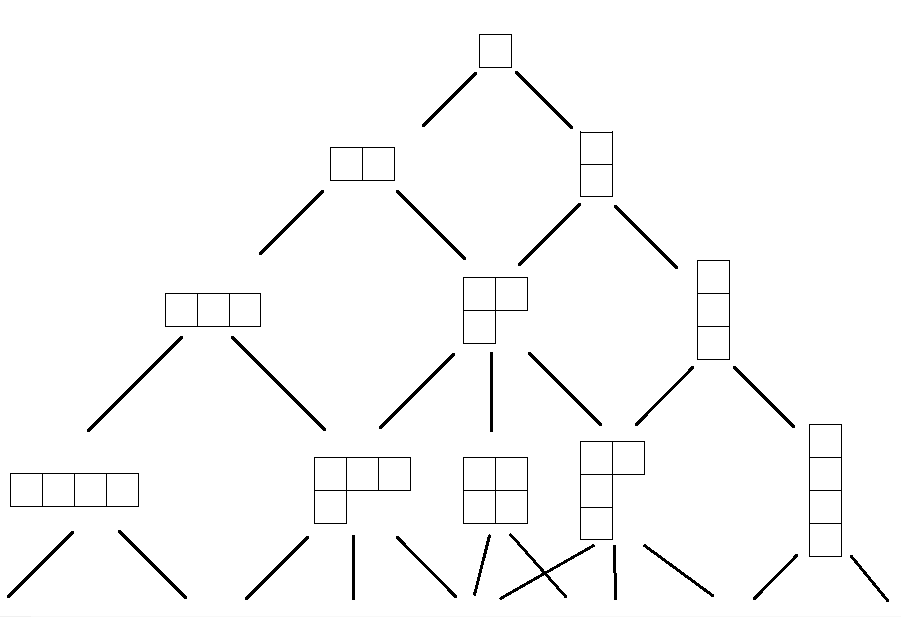} 
\caption{The Young graph. Each path from the top node to a particular Young diagram $\lambda$ is identified with a standard Young tableau of shape $\lambda$.
}
\label{Young graph}
\end{center}
\end{figure}

\section{Decomposition of the function space on the Johnson graph}
\label{Johnson graph}

 We define a  \textit{$k$-set} as a subset of $\{1,...,n\}$ of cardinality $k$. Let   $X$ be the set of all $k$-sets.  Given two $k$-sets $x,y$ the distance  $d(x,y)$  is defined as $n- \lvert x \cap y \rvert$. The group $S_n$ acts naturally on  $X$ by
$$
\sigma \{i_1,...,i_k\} = \{\sigma(i_1),...,\sigma(i_k)\}
$$
The vector space $\mathcal{F}$ of the complex valued functions on $X$ is a complex representation of  $S_n$ where the action is given by $\sigma f = f \circ \sigma^{-1}$.

 To each $k$-set $x \in X$ we attach the delta function $\delta(x)$ defined on $X$ by 
$$
\delta(x) (z) = 
\left\{ 
\begin{array}{cc}
1& \mbox{if} \ \  x=z \\
0 & \mbox{otherwise}
\end{array}
\right. .
$$

We consider $\mathcal{F}$ as an inner product space where the inner product is such that the delta functions form an orthonormal basis.

A Young diagram can be identified with the sequence given by the numbers of boxes in the rows, written top down. For example the Young diagram
\ytableausetup{smalltableaux}
\begin{center}
\ydiagram{5,4,2}
\end{center}
is identified with $(5,4,2)$.
It can be shown (see \cite{Ceccherini}, \cite{Stanton}) that the decomposition of $\mathcal{F}$ as a direct sum of irreducible representations of $S_n$ is given as follows.

\begin{thm}
\label{decomposition of F}
Let $s=min(k,n-k)$. The space  $\mathcal{F}$ of functions on the Johnson graph J(n,k) decomposes in  $s+1$ multiplicity-free irreducible representations of the group $S_n$. Moreover, the decomposition is given by
$$
\mathcal{F} = \bigoplus_{i=0}^{s} V_{\alpha_i}
$$ 
where $\alpha_i$ is the Young diagram $(n-i,i)$.
\end{thm}

For example, if $n=6$ and $k=2$ then  the irreducible components of $\mathcal{F}$ are in correspondence with the Young diagrams
\ytableausetup{smalltableaux}
\begin{center}
 \ \ 
\ydiagram{6,0} \  \ \ \ 
\ydiagram{5,1} \  \ \ \ 
\ydiagram{4,2} 
\end{center}

From now on we denote by $s$ the number $min(k,n-k)$.

\subsection{Gelfand-Tsetlin basis of $\mathcal{F}$}
\label{Gelfand-Tsetlin basis of F}

From Theorem \ref{decomposition of F} we see that $\mathcal{F}$ has a well-defined --up to scalars-- Gelfand-Tsetlin basis and that there is a bijection between the set of elements of this GT-basis and the set of standard tableaux of shape $(n-a,a)$ where $a$ runs from $0$ to $s$.

Let us give a more explicit description of the GT-basis of $\mathcal{F}$. 
Consider the space $ \mathcal{F}$ as an $S_i$-module for $i=1,...,n$, and 
 let 
 $\mathcal{F}_{i,\lambda} $ be the isotypic component corresponding to the irreducible representation $\lambda$ of $S_i$ so that for each $i$ we have a decomposition
 $$
 \mathcal{F}= \bigoplus_{\lambda \in Irr(S_i)} \mathcal{F}_{i,\lambda}
 $$
 where $\mathcal{F}_{i,\lambda} \perp \mathcal{F}_{i,\lambda'}$ if $\lambda \neq \lambda'$.
 
 If $\lambda_1, \lambda_2, ...., \lambda_n$ is a sequence of Young diagrams where $\lambda_i$ corresponds to a representation of $S_i$ we define
$$
 \mathcal{F}_{\lambda_1  \lambda_2  ...  \lambda_n} 
 =
  \mathcal{F}_{1,\lambda_1} \cap \mathcal{F}_{2,\lambda_2} \cap ... \cap
 \mathcal{F}_{n,\lambda_n}
$$
 From the branching rule for representations of $S_n$ and from Theorem \ref{decomposition of F} it turns out that $\mathcal{F}$ has an orthogonal decomposition in one-dimensional subspaces
$$
 \mathcal{F} =  \bigoplus_{\lambda_1 \nearrow \lambda_2 \nearrow ... \nearrow \lambda_n}
 \mathcal{F}_{\lambda_1  \lambda_2  ...  \lambda_n}
 $$
where $\lambda_n$ runs through all representations of $S_n$ corresponding to Young diagrams $(n-a,a)$ for $a=0,...,s$ (see Figure \ref{GT of the Johnson graph}).

\begin{figure}[ht] 
\begin{center}
\includegraphics[scale=0.45]{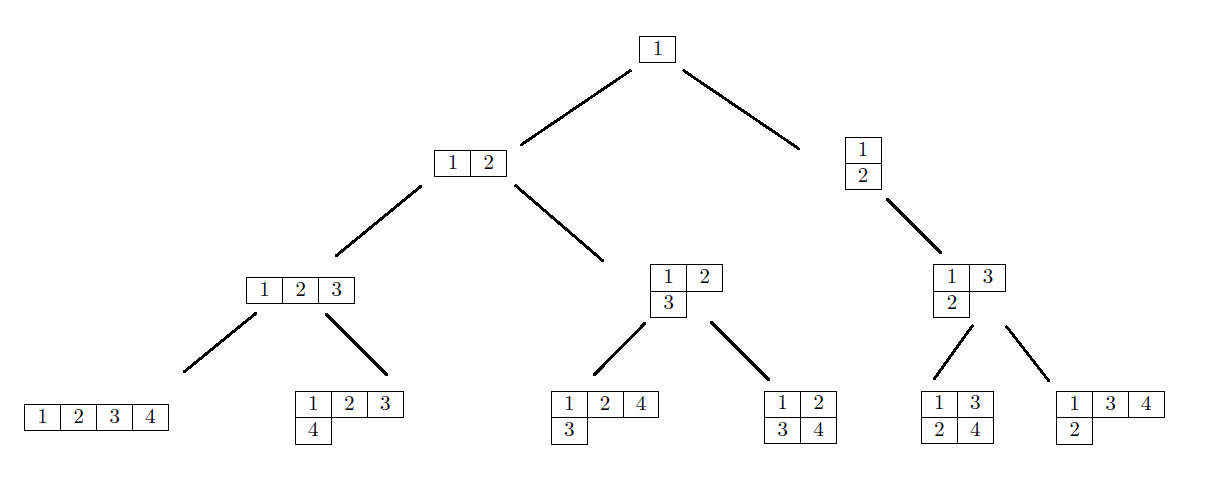} 
\caption{The leaves of this tree parametrize the Gelfand-Tsetlin basis of the space $\mathcal{F}$ of functions on the Johnson graph $J(4,2)$.
}
\label{GT of the Johnson graph}
\end{center}
\end{figure}

\section{Adapted decompositions of $\mathcal{F}$}
\label{Decomposition theorems}

We represent a $k$-set by a word in the alphabet $\{1,2\}$ as follows. The element $i \in \{1,...,n\}$ belongs to the $k$-subset if and only if the place $i$ of the  word is occupied by the letter $1$. For example, 
$$
\{2,3,6,8\} \subseteq \{1,...,9\} \ \ \ \ \ \rightarrow \ \ \ \ \ \  2\ 1\ 1\ 2\ 2\ 1\ 2\ 1\ 2
$$
So, from now on, we identify $X$ with the set of words of length $n$ in the alphabet $\{1,2\}$ such that the letter $1$ appears $k$ times.  The group $S_n$ acts on $X$ in the natural way.

For $i= 1,...,n$ and $c \in \{1,2\} $ we define $\mathcal{F}^{i,c}$ as the subspace of $\mathcal{F}$ generated by the delta functions $\delta(x)$ such that the word $x$ has the letter $c$ in the place $i$.
For each $i$ we have a decomposition 
 $$
 \mathcal{F}=  \mathcal{F}^{i,1} \oplus \mathcal{F}^{i,2}
 $$
 with $ \mathcal{F}^{i,1} \perp \mathcal{F}^{i,2}$.

 \begin{definition}
 For $q=1,...,n$ let 
 $c_{q} c_{q+1} ... c_n$ be a word whose letters are in the alphabet $\{1,2\}$. For $q=n+1$ the word $c_{q} c_{q+1} ... c_n$ denotes the word with no letters.
 For $q=1,...,n$ we define $\mathcal{F}^{c_{q} c_{q+1} ... c_n}$ as the subspace of $\mathcal{F}$
 $$
  \mathcal{F}^{c_{q} c_{q+1} ... c_n} = \mathcal{F}^{q,c_{q}}  \cap 
 \mathcal{F}^{q+1,c_{q+1}} \cap ... \cap \mathcal{F}^{n,c_{n}}
$$
In the case $q=n+1$ we set  $\mathcal{F}^{c_{q} c_{q+1} ... c_n} = \mathcal{F}$
\end{definition}
We see that $ \mathcal{F}^{c_{q}  ... c_n} $ is nontrivial if and only if the number of letters $1$ in 
 the word $c_{q} ... c_n$ is at most $k$. 
 
 \begin{definition}
 For $q=1,...,n$ we define $X^{c_q...c_n}$ as the subset of those words $w_1...w_n$ in $X$ such that $w_q...w_n$ is $c_q...c_n$. For $q=n+1$ we set $X^{c_q...c_n}=X$
 \end{definition}

 Observe that each subset $X^{c_q...c_n}$ is stabilized by the action of the subgroup $S_{q-1}$. 
 We have  
 $$
 \mathcal{F}^{c_q...c_n} = \bigoplus_{w_1...w_n \in X^{c_q...c_n} } \mathbb{C} \ \ \delta(w_1...w_n)
 $$ 
 
 \begin{definition}
 For $q=1,...,n$ we define $X_q$ as the set of words $c_q...c_n$ where the number of letters $1$ is at most $k$. For $q=n+1$ we set $X_q$ to be the set whose only element is the word with no letters.
 \end{definition}

 Then $\mathcal{F}$ decomposes as
\begin{equation} \label{F}
 \mathcal{F} = \bigoplus_{c_q...c_n \in X_q} \mathcal{F}^{c_q...c_n} 
\end{equation}
 and each subspace $ \mathcal{F}^{c_q...c_n}$ is invariant by the action of $S_{q-1}$.
 
\begin{definition} 
For $1 \leq p  < q \leq n+1$ we define
$$
 \mathcal{F}_{\lambda_1...\lambda_p}^{c_{q} ... c_n} =  \mathcal{F}_{\lambda_1...\lambda_p} \cap   \mathcal{F}^{c_{q} ... c_n} 
$$
\end{definition}

\begin{thm} \label{One-dimensional decomposition}
For fixed $1 \leq q \leq n+1$, the space $\mathcal{F}$ decomposes in one-dimensional subspaces as 
$$
\mathcal{F} =  \bigoplus  \mathcal{F}_{\lambda_1...\lambda_{q-1}}^{c_q ... c_n} 
$$
where the direct sum runs through all $\lambda_1...\lambda_{q-1}$  and $c_q ... c_n$ such that  if the number of letters $1$ in the word $c_q ... c_n$ is $k-r$, where $0 \leq r \leq k$, then the sequence $\lambda_1...\lambda_{q-1}$ forms a standard Young tableaux $\lambda_1 \nearrow \lambda_2 \nearrow ... \nearrow \lambda_{q-1}$ where $\lambda_{q-1}$ is a Young diagram of the form $(q-1-a',a')$ with $a'=0,...,min(r,q-1-r)$.
\begin{proof}
Suppose that the letter $1$ appears $k-r$ times in the word $c_q...c_n$, where $0 \leq r \leq k$. Then the subset $X^{c_q...c_n}$ consists of those words  $w_1...w_n$ such that  $w_1...w_{q-1}$  has exactly $r$ appearances of the letter $1$. This means that $X^{c_q...c_n}$ has the structure of the Johnson graph $J(q-1,r)$ and, when acted by the subgroup $S_{q-1}$, the space of $\mathbb{C}$-valued functions on $X^{c_q...c_n}$ decomposes as an $S_{q-1}$-module in a multiplicity-free way according to the  formula of Theorem \ref{decomposition of F}.
 As a consequence, each  subspace $ \mathcal{F}^{c_q...c_n}$  has a Gelfand-Tsetlin decomposition
\begin{equation} \label{1}
  \mathcal{F}^{c_q...c_n}  =  \bigoplus_{\lambda_1 \nearrow \lambda_2 \nearrow ... \nearrow \lambda_{q-1}}
 \mathcal{F}^{c_q...c_n}_{\lambda_1  ...  \lambda_{q-1}},
\end{equation}
 where $\lambda_{q-1}$ runs over all Young diagrams $(q-1-a' , a')$ with $0 \leq a' \leq min(r,q-1-r)$.
 Then the theorem follows from (\ref{F}).
\end{proof}
\end{thm}

\begin{thm} \label{Adapted decomposition}
For fixed $1 \leq p  < q \leq n+1$, the space $\mathcal{F}$ has a orthogonal decomposition 
$$
\mathcal{F} =  \bigoplus  \mathcal{F}_{\lambda_1...\lambda_p}^{c_q ... c_n} 
$$
where the direct sum runs through all $\lambda_1...\lambda_p$  and $c_q ... c_n$ such that  if the number of letters $1$ in the word $c_q ... c_n$ is $k-r$, where $0 \leq r \leq k$, then the sequence $\lambda_1...\lambda_p$ forms a standard Young tableaux $\lambda_1 \nearrow \lambda_2 \nearrow ... \nearrow \lambda_p$ where $\lambda_p$ is a Young diagram of the form (p-a,a) with $a=0,...,min(r,q-1-r, p/2)$.
\end{thm}

\begin{proof}
  For fixed $c_q...c_n$, we group the one dimensional subspaces $\mathcal{F}^{c_q...c_n}_{\lambda_1  ...  \lambda_{q-1}}$ in Equation (\ref{1}) according to the $p$ initial Young diagrams defining the standard Young tableaux $\lambda_1  ...  \lambda_{q-1}$ and obtain 
 $$
 \mathcal{F}^{c_q...c_n}_{\lambda_1  ...  \lambda_{p}} = \bigoplus_{\lambda_{p+1} \nearrow \lambda_{p+2} \nearrow ... \nearrow \lambda_{q-1}} 
 \mathcal{F}^{c_q...c_n}_{\lambda_1  ...  \lambda_{q-1}},
 $$
 where the direct sum runs through all Young diagrams $\lambda_{p+1} \lambda_{p+2}  ...  \lambda_{q-1}$ such that $\lambda_1 \nearrow \lambda_2 \nearrow ... \nearrow \lambda_{q-1}$ is a standard Young tableau where $ \lambda_{q-1}$ is a Young diagram of the form $(q-1-a' , a')$ with $0 \leq a' \leq min(r,q-1-r)$.
 
Then  we obtain the decomposition
\begin{equation} \label{Fc}
  \mathcal{F}^{c_q...c_n}  =  \bigoplus_{\lambda_1 \nearrow \lambda_2 \nearrow ... \nearrow \lambda_{p}}
 \mathcal{F}^{c_q...c_n}_{\lambda_1  ...  \lambda_{p}},
\end{equation}
where $\lambda_1  ...  \lambda_{p}$ run through all sequences of Young diagrams for which there is a sequence $\lambda_{p+1}  ...  \lambda_{q-1}$ such that $\lambda_1 \nearrow \lambda_2 \nearrow ... \nearrow \lambda_{q-1}$ is a standard Young tableau of the form $(q-1-a' , a')$ with $0 \leq a' \leq min(r,q-1-r)$. Such sequences $\lambda_1  ...  \lambda_{p}$ are characterized as those sequences such that $\lambda_1 \nearrow \lambda_2 \nearrow ... \nearrow \lambda_{p}$ is a standard Young tableau of the form $(p-a,a)$ with $0 \leq a \leq min(r,q-1-r, p/2)$. From (\ref{F})  and (\ref{Fc}) we obtain the theorem.
 
 \end{proof}

\section{The intermediate bases}
\label{Intermediate}

Let us describe schematically the Fast Fourier Transform algorithm for the Johnson graph. The input is a vector $f$ in the space $\mathcal{F}$ of functions on the set $X$ of $k$-sets, written in the delta function basis $B_0$, given as a column vector $[f]_{B_0}$. The output of the algorithm is a column vector representing the  vector $f$ written in the basis $B_n$, the Gelfand-Tsetlin basis of $\mathcal{F}$. In other words, the objective is to apply  the change of basis matrix to a given column vector:
$$
[f]_{B_n} = [B_0]_{B_n} \  [f]_{B_0}
$$
Our technique to realize this matrix multiplication is to construct a sequence of intermediate orthonormal bases $B_1, B_2, ..., B_{n-1}$ such that 
$$
[B_0]_{B_n} =  [B_{n-1}]_{B_n}  \  \ ...\ \ [B_1]_{B_2} \ \  [B_0]_{B_1}\
$$
is a decomposition where each factor is a sparse matrix.

\subsection{Block-diagonal matrices}
In order to establish the sparsity of a matrix we will rely on the following simple principle given by Lemma \ref{Block-diagonal}.

Whenever $T$ is a linear operator on a vector space $V$ and $B$ is a base of $V$ we denote by $[T]_{B}$ the matrix whose columns are the elements of the base $B$ transformed by $T$ and written in the base $B$.

 \begin{definition}
 Let $V$ be a finite-dimensional vector space. For $i=1,...,n$ let $V_i$ be a subspace of $V$ such that
 $$
 V = \bigoplus_{i=1}^{n} V_i
 $$
 Let $B$ be a basis of $V$ and $T$ a linear endomorphism of $V$. We say that the basis $B$ is adapted to the decomposition if every element of $B$ is in $V_i$ for some $i$.  We say that $T$ is adapted to the decomposition if it preserves each subspace $V_i$.
\end{definition}

We will use the following simple fact from linear algebra.
\begin{lemma} \label{Block-diagonal}
Let $V$ be a finite-dimensional vector space with a direct sum decomposition $V = \bigoplus_{i=1}^{n} V_i$. Let $B$ and $B'$ be two bases of $V$, both adapted to this decomposition and let $T$ be a linear endomorphism of $V$ adapted to this decomposition. Then

a) there are orders of the elements of $B$ and $B'$ such that the change of basis matrix $[B]_{B'}$ is block-diagonal, with each block of size $dim(V_i)$ for each $i=1,...,n$.

b) there is an order of the basis $B$  such that the $[T]_{B}$ is block-diagonal, with a block of size $dim(V_i)$ for each $i=1,...,n$.
\end{lemma}

 \subsection{Definition of the basis $B_i$}

Let $0 \leq i \leq n$. As observed in the proof of Theorem \ref{Adapted decomposition}, the set  $X^{c_{i+1}...c_n}$ has the structure of the Johnson graph $J(i,r)$ where $k-r$ is the number of letters $1$ in the word $c_{i+1}...c_n$, and when acted by the subgroup $S_i$, the space $\mathcal{F}^{c_{i+1}...c_n}$ decomposes as an $S_i$-module in a multiplicity-free way according to the  formula of Theorem \ref{decomposition of F}. As a consequence, each subspace $\mathcal{F}^{c_{i+1}...c_n}$ has a Gelfand-Tsetlin decomposition into one-dimensional subspaces
\begin{equation}  \label{Pre-definition of the basis}
  \mathcal{F}^{c_{i+1}...c_n}  =  \bigoplus_{\lambda_1 \nearrow \lambda_2 \nearrow ... \nearrow \lambda_{i}}
 \mathcal{F}^{c_{i+1}...c_n}_{\lambda_1  ...  \lambda_{i}},
\end{equation}
 where $\lambda_{i}$ runs over all Young diagrams $(i-a , a)$ with $0 \leq a \leq min(r,i-r)$. 
 Let $B^{c_{i+1}...c_{n}}$ be the unique, up to scalars, basis of $\mathcal{F}^{c_{i+1}...c_n}$ adapted to the decomposition  (\ref{Pre-definition of the basis}).
 Then the space $\mathcal{F}$ decomposes in one-dimensional subspaces as
 \begin{equation} \label{Definition of the basis}
  \mathcal{F} = \bigoplus_{c_{i+1}...c_n \in X_i} \ \ \bigoplus_{\lambda_1 \nearrow \lambda_2 \nearrow ... \nearrow \lambda_{i}}
 \mathcal{F}^{c_{i+1}...c_n}_{\lambda_1  ...  \lambda_{i}},
\end{equation}
 \begin{definition}
 We define the $i$-th intermediate basis of $\mathcal{F}$ as the unique, up to scalars, basis $B_i$ of $\mathcal{F}$ adapted to the decomposition (\ref{Definition of the basis}).
  \end{definition}
  We have 
 $$
 B_i = \bigsqcup_{c_{i+1}...c_n \in X_i}  B^{c_{i+1}...c_n}
 $$

From  Theorem \ref{decomposition of F}, we see that the basis $B^{c_{i+1}...c_n}$ is parametrized by the set of  standard tableaux of shape $(i-a , a)$ with $0 \leq a \leq min(r,i-r)$. On the other hand, the word $c_{i+1}...c_n$ runs over the set $X_i$. Figure \ref{Intermediate bases} illustrates the structure of the intermediate bases.

\begin{figure}[ht] 
\begin{center}
\includegraphics[scale=0.6]{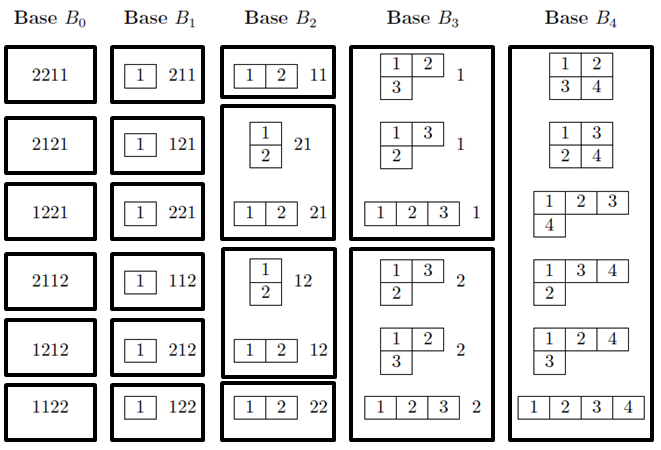} 
\caption{Labels of the intermediate bases in the case $n=4, k=2$.  The boxes in each column represent the decomposition $B_i = \sqcup B^{c_{i+1}...c_n}$.
}
\label{Intermediate bases}
\end{center}
\end{figure}

 \begin{lemma} \label{Adapted bases}
 For $1 \leq p \leq i < q \leq n+1$, the basis $B_i$ is adapted to the decomposition of Theorem \ref{Adapted decomposition}
 $$
\mathcal{F} =  \bigoplus  \mathcal{F}_{\lambda_1...\lambda_p}^{c_q ... c_n} 
$$
\end{lemma}
 \begin{proof}
 It is clear from the definition of the subspace  $\mathcal{F}_{\lambda_1...\lambda_p}^{c_q ... c_n} $ that if 
 $1 \leq p \leq p' < q' \leq q \leq n+1$ then 
 $$
  \mathcal{F}_{\lambda_1...\lambda_{p'}}^{c_{q'} ... c_n}   \subseteq  
    \mathcal{F}_{\lambda_1...\lambda_p}^{c_q ... c_n} 
 $$
 Since every element of the basis $B_i$ is in one of the subspaces of the form 
 $\mathcal{F}_{\lambda_1...\lambda_i}^{c_{i+1} ... c_n}$
 the Lemma follows from the case $p'=i$, $q'=i+1$.
 
 \end{proof}

 \subsection{Sparsity of the change of basis matrix $[B_i]_{B_{i-1}}$}
 
In this section we establish the fact that, for all $i$, each column of the matrix $[B_i]_{B_{i-1}}$ has at most two nonzero entries. In fact, we show that if the bases are properly ordered, the matrix $[B_i]_{B_{i-1}}$ is block-diagonal with blocks of size at most two.

\begin{thm}
\label{Sparsity of matrix}
There is an order of the basis $B_i$ and an order of the basis $B_{i-1}$ such that the change of basis matrix $[B_i]_{B_{i-1}}$ is block-diagonal with all blocks  of  size  at most two.

 \begin{proof}
 By Lemma \ref{Adapted bases}, both $B_i$ and $B_{i-1}$ are adapted to the decomposition
 $$
 \mathcal{F} = \bigoplus \mathcal{F}^{c_{i+1}...c_n}_{\lambda_1  ...  \lambda_{i-1}}
 $$
 that is, the decomposition of Theorem \ref{Adapted decomposition} with $p=i-1$ and $q=i+1$.
 Observe that the subspace $\mathcal{F}^{c_{i+1}...c_n}_{\lambda_1  ...  \lambda_{i-1}}$
 has dimension at most two, since it is spanned by the  subspaces 
 $\mathcal{F}^{1 c_{i+1}...c_n}_{\lambda_1  ...  \lambda_{i-1}}$ and $\mathcal{F}^{2 c_{i+1}...c_n}_{\lambda_1  ...  \lambda_{i-1}}$ and --according to Theorem \ref{One-dimensional decomposition}-- these subspaces have dimension at most one. Then the theorem follows from Lemma \ref{Block-diagonal}.
 \end{proof}
\end{thm}

\begin{thm}
\label{Main theorem}
Let $B_0$ be the delta function basis of $\mathcal{F}$ and let $B_n$ be a Gelfand-Tsetlin basis of $\mathcal{F}$. We assume that the matrices $[B_{i-1}]_{B_{i}}$ for $i=2,3,...,n$ have been computed. Then, given a column vector $[f]_{B_0}$ with $f \in \mathcal{F}$, the column vector $[f]_{B_n}$ given by 
$$
[f]_{B_n} = [B_0]_{B_n} \  [f]_{B_0}
$$
can be computed using at most $2(n-1) \binom{n}{k}$ operations.
\end{thm}
\begin{proof}
By Theorem  \ref{Sparsity of matrix} we see that each column of the matrix 
$[B_{i-1}]_{B_{i}}$ has at most two non-zero elements, no matter the order of each basis. Then the application of  the matrix $[B_{i-1}]_{B_{i}}$ to a generic column vector can be done using at most $2 \binom{n}{k}$ operations. 
Observe that $[B_0]_{B_1}$ is the identity matrix. We have 
\begin{eqnarray*}
[B_0]_{B_n} &= [B_{n-1}]_{B_{n}}\ \ ...\ \  [B_{1}]_{B_{2}} \ \ [B_{0}]_{B_1}  =[B_{n-1}]_{B_{n}}\ \ ...\ \  [B_{1}]_{B_{2}}  \end{eqnarray*}
Then the successive applications of the $n-1$ matrices can be done in at most $2(n-1) \binom{n}{k}$ operations.
\end{proof}

\subsection{Example}

Consider the case $n=4$, $k=2$. For each vector $b$ of the basis $B_{i}$, there exists a unique word  ${c_{i+1}...c_n} \in X_i$ and a unique standard tableau $\lambda_1 \nearrow \lambda_2 \nearrow ... \nearrow \lambda_i$ such that $b \in \mathcal{F}^{c_{i+1}...c_n}_{\lambda_1  ...  \lambda_{i}}$. Then $b$ is a linear combination of those elements of $B_{i-1}$ that belong to the spaces $\mathcal{F}^{1c_{i+1}...c_n}_{\lambda_1  ...  \lambda_{i-1}}$ and $\mathcal{F}^{2c_{i+1}...c_n}_{\lambda_1  ...  \lambda_{i-1}}$. Then the matrices $[B_i]_{B_{i-1}}$  have the form

\[
\left[  B_{1}\right]  _{B_{2}}=\left[
\begin{array}
[c]{cccccc}%
\ast & 0 & 0 & 0 & 0 & 0\\
0 & \ast & \ast & 0 & 0 & 0\\
0 & \ast & \ast & 0 & 0 & 0\\
0 & 0 & 0 & \ast & \ast & 0\\
0 & 0 & 0 & \ast & \ast & 0\\
0 & 0 & 0 & 0 & 0 & \ast
\end{array}
\right]  \text{ \ \ \ }\left[  B_{2}\right]  _{B_{3}}=\left[
\begin{array}
[c]{cccccc}%
\ast & 0 & \ast & 0 & 0 & 0\\
0 & \ast & 0 & 0 & 0 & 0\\
\ast & 0 & \ast & 0 & 0 & 0\\
0 & 0 & 0 & \ast & 0 & 0\\
0 & 0 & 0 & 0 & \ast & \ast\\
0 & 0 & 0 & 0 & \ast & \ast
\end{array}
\right]
\]
\\
\[
\left[  B_{3}\right]  _{B_{4}}=\left[
\begin{array}
[c]{cccccc}%
\ast & 0 & 0 & 0 & \ast & 0\\
0 & \ast & 0 & \ast & 0 & 0\\
0 & 0 & \ast & 0 & 0 & \ast\\
0 & \ast & 0 & \ast & 0 & 0\\
\ast & 0 & 0 & 0 & \ast & 0\\
0 & 0 & \ast & 0 & 0 & \ast
\end{array}
\right].
\]

\begin{figure}[ht] 
\begin{center}
\includegraphics[scale=0.5]{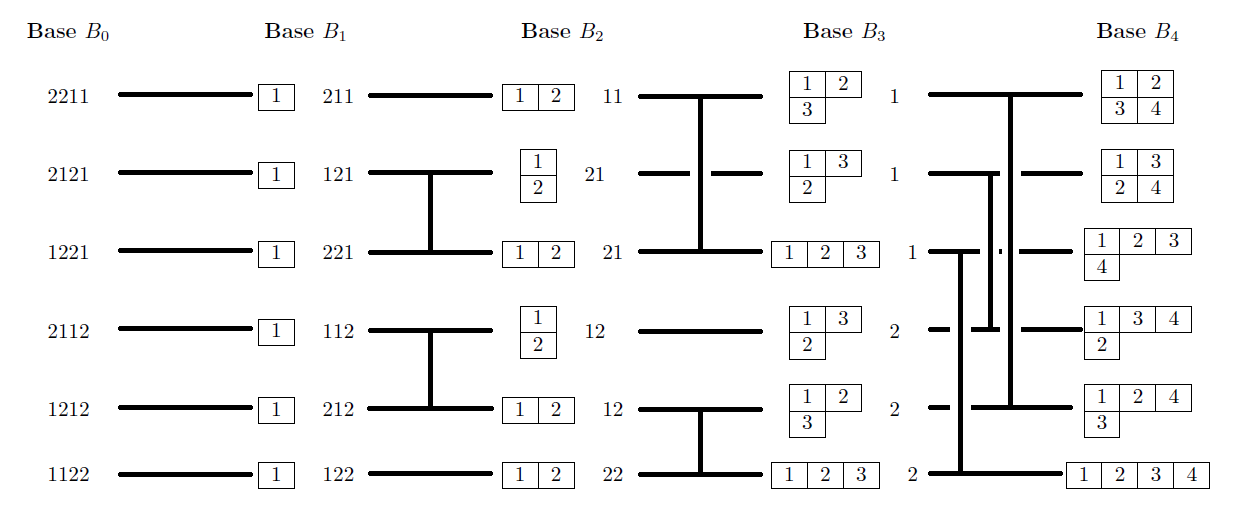} 
\caption{An illustration of the sparsity of the matrices $[B_{i-1}]_{B_{i}}$. A label of an element $b \in B_{i-1}$ is  connected with a label of an element $b' \in B_{i}$ if and only if they are S-related.  Two labels are horizontally adyacent if and only if they are RS-related, that is, each row corresponds to the process given by the Robinson-Schensted insertion algorithm.}
\label{Sparsity}
\end{center}
\end{figure}

\section{Connection with the Robinson-Schensted insertion algorithm}
\label{RS}

In Figure \ref{Sparsity} the vertical order of the labels of the elements of each basis $B_i$ has been carefully chosen in order to simplify the figure. In fact, the order is such that each horizontal line corresponds to a well known process: the Robinson-Schensted (RS) insertion algorithm (see  \cite{Fulton}). 

Observe that each horizontal line gives the sequence --reading from left to right-- that is obtained by applying the RS insertion algorithm to a word corresponding to an element of the basis $B_0$, which is a word in the alphabet $\{1,2\}$. The elements of this sequence are triples $(P,Q,\omega)$ where $P$ is a semistandard tableau, $Q$ is a standard tableau and $\omega$ is a word in the alphabet $\{1,2\}$. In our situation $P$ is filled with letters in $\{1,2\}$ so its height is at most $2$. It turns out that the triple $(P,Q,\omega)$ is determined by the pair $(Q,\omega)$ so $P$ can be ommited.

\begin{definition}
For $1 \leq i \leq n$, let $b \in B_{i-1}$ and $b' \in B_{i}$. We say that $b$ and $b'$ are \textit{S-related} if both belong to the subspace 
$
 \mathcal{F}^{c_{i+1}...c_n}_{\lambda_1  ...  \lambda_{i-1}}
$
for some standard Young tableau $\lambda_1 \nearrow \lambda_2 \nearrow ... \nearrow \lambda_{i-1}$ and some word $c_{i+1},..., c_n$.
\end{definition}

\begin{definition}
Let $b \in B_{i-1}$ and $b' \in B_{i}$. We say that $b$ and $b'$ are \textit{RS-related} if the label of $b'$ is obtained by applying the RS insertion step to the label of $b$. 
\end{definition}

From the definitions it is immediate the following (see Figure \ref{Sparsity} for an illustration).
\begin{thm}
If $b \in B_{i-1}$ and $b' \in B_{i}$ are RS-related then they are S-related.
\end{thm}

\section{Jucys-Murphy operators}
\label{Jucys-Murphy operators}

Let $\mathbb{C}(S_n)$ denote the group algebra of the group $S_n$. Let $(ij)$ denote the transposition that interchanges $i$ with $j$. For $i=1,2,...,n$, the \textit{Jucys-Murphy element} $J_i$ is defined as the element of $\mathbb{C}(S_n)$ given by 
$$
J_i = (1i) + (2i) + ... + ((i-1)i)
$$
(in particular $J_1=0$).
For $i=1,2,...,n$ let $Z_i \in \mathbb{C}(S_n)$ be 
$$
Z_i = \text{sum of all transpositions in}  \ S_i
$$
Observe that for $i \in \{2,...,n\}$
$$
J_i = Z_i - Z_{i-1}.
$$
If $V$ is a representation of $S_n$ we consider the canonical $\mathbb{C}(S_n)$-module structure on $V$ and we identify an element $A$ of $\mathbb{C}(S_n)$ with the linear operator  $A: V \rightarrow V$ sending $v$ to $Av$. 

\begin{prop} 
\label{Basis elements are eigenvectors of JM operators}
Let $V$ be a multiplicity-free representation of $S_n$. Then every element of a Gelfand-Tsetlin basis of $V$ is an eigenvector of $J_i$ for all $i \in \{1,...,n\}$.
\end{prop}
\begin{proof}
Let $b$ be an element of a Gelfand-Tsetlin basis of $V$. For $i=1$ the proposition is trivial so let $i$ be any element of $\{2,...,n\}$. Then the vector $b$ belongs to some isotypic component $H_i$ of the decomposition of $V$ as a representation of $S_i$. 

We will use the following characterization of isotypic components. Let $End(V)^{S_i}$ be the ring of interwining operators, that is, linear operators $Z \in End(V)$ such that $ZA=AZ$ for all $A\in \mathbb{C}(S_i)$. Observe that $V$ has a natural $End(V)^{S_i}$-module structure.
Then, a subspace of $V$ is an isotypic component of the action of $S_i$ if and only if it is a minimal element in the lattice of subspaces that are simultaneously a $\mathbb{C}(S_i)$-submodule and a $End(V)^{S_i}$-submodule of $V$. 

Since $H_i$ is an isotypic component it is a $\mathbb{C}(S_i)$-submodule of $V$. Since $Z_i \in \mathbb{C}(S_i)$ we see that $Z_i(H_i) \subseteq H_i$ . Let $H'_i$ be an eigenspace of the restriction $Z_i : H_i \rightarrow H_i$ with eigenvalue $\alpha_i$. Since $Z_i \in End(V)^{S_i}$, for any $A\in \mathbb{C}(S_i)$ and any $h \in H'_i$ we have
$$
Z_i(A(h)) = A(Z_i(h)) = \alpha_i A(h)
$$
This shows that  $H'_i$ is a $\mathbb{C}(S_i)$-submodule of $V$.

On the other hand, since $Z_i $ belongs to  $\mathbb{C}(S_i)$ it commutes with every element of $End(V)^{S_i}$. 
Then for any $Z \in End(V)^{S_i}$ and any $h \in H'_i$ we have
$$
Z_i(Z(h)) = Z(Z_i(h)) = \alpha_i Z(h).
$$
This shows that $H'_i$ is a $End(V)^{S_i}$-submodule of $V$.
Then $H'_i$ is simultaneously a $\mathbb{C}(S_i)$-submodule and a $End(V)^{S_i}$-submodule of $V$, and it is contained in $H_i$. Since $H_i$ is an isotypic component, it is minimal among subspaces with this property, then $H'_i=H_i$. This proves that $H_i$ is an eigenspace of $Z_i$. We have $b \in H_i$, then $b$ is an eigenvector of $Z_i$ for all $i \in \{2,...,n\}$. Since $J_i = Z_i - Z_{i-1}$, we see that $b$ is also an eigenvector of $J_i$ with eigenvalue $\alpha_i - \alpha_{i-1}$ for all $i \in \{2,...,n\}$.
\end{proof}

\begin{cor} \label{Eigenvector}
For $1 \leq i \leq n$, the one-dimensional subspace $\mathcal{F}_{\lambda_1...\lambda_{i}}^{c_{i+1} ... c_n} $  is invariant by the action of $J_j$ for $j=1,...,i$.
\end{cor}
\begin{proof}
From (\ref{1}) we have that $\mathcal{F}^{c_{i+1} ... c_n} $ is a multiplicity-free representation of $S_i$ such that the Gelfand-Tsetlin decomposition is given by the subspaces $\mathcal{F}_{\lambda_1...\lambda_{i}}^{c_{i+1} ... c_n} $. From Proposition \ref{Basis elements are eigenvectors of JM operators} we have that these subspaces are formed by eigenvectors of $J_j$ for $j=1,...,i$.
\end{proof}

Let $V$ be a multiplicity-free representation of $S_n$ and $B$ a Gelfand-Tsetlin basis of $V$. By Proposition \ref{Basis elements are eigenvectors of JM operators} there is a map $\alpha : B \rightarrow \mathbb{C}^n$ 
$$
\alpha(b)= (\alpha_1(b),...,\alpha_n(b))
$$
where $\alpha_i(b)$ is defined by 
$$
J_i b = \alpha_i(b) \ b
$$

In \cite{Vershik}  --Proposition 5.3 together with Theorem 5.8-- this map is completely described by giving an explicit formula for the eigenvalues $\alpha_i(b)$ as follows. 
Let $\lambda_1 \nearrow \lambda_2 \nearrow ... \nearrow \lambda_n$ be the chain corresponding to the basis element $b$, that is, $b$ is in the one-dimensional subspace
$$
V_{1,\lambda_1}  \cap ... \cap  V_{n,\lambda_{n}} 
$$
where $V_{i,\lambda_i}$ is the isotypic component of the action of the subgroup $S_i$ on $V$ corresponding to the representation $\lambda_i$. We identify each $\lambda_i$ with its corresponding Young diagram and set $\lambda_0 = \emptyset$. The Young diagram $\lambda_i$ is obtained from $\lambda_{i-1}$ by adding a single box 
$$\square = \lambda_{i} / \lambda_{i-1}$$
The content of a box in a Young diagram is defined as 
$$
c(\square) = x\mbox{-coordinate of}\  \square\  \ - \  y\mbox{-coordinate of}\  \square
$$

\begin{figure}[ht] 
\begin{center}
\includegraphics[scale=0.3]{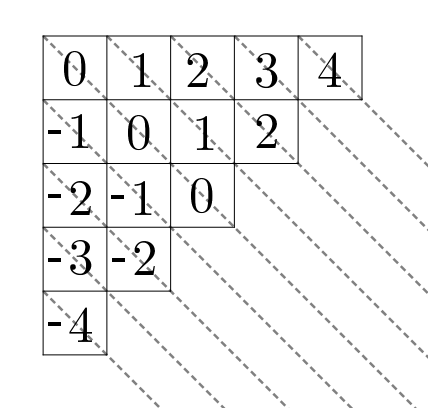} 
\caption{A Young diagram where each box is filled with its content.}
\label{Contents}
\end{center}
\end{figure}

According to the main theorem in \cite{Vershik}, the eigenvalue $\alpha_i(b)$ is given by
\begin{align} \label{eigenvalues}
\alpha_i(b) = c(\lambda_{i} / \lambda_{i-1})
\end{align}
and the map $\alpha$ is a bijection between the Gelfand-Tsetlin base $B$ of $V$ and the vectors in $\mathbb{C}^n$ of the form
$$
( c(\lambda_{1} / \lambda_{0}),...,c(\lambda_{n} / \lambda_{n-1})).
$$
where $\lambda_1  \nearrow ... \nearrow \lambda_n$ is a standard Young tableau with $\lambda_n$ in the decomposition of $V$.
For example, the map $\alpha$ gives the correspondence
$$\young(134,25) \ \ \ \ \ \  \longleftrightarrow \ \ \ \ \ \ (0,-1,1,2,0)  $$

\section{Efficient computation of the matrices ${[B_i]}_{B_{i-1}}$}
\label{Efficient computation of the matrices}

We have shown that the Fourier transform on the Johnson graph can be performed by the succesive applications of the matrices ${[B_{i-1}]}_{B_{i}}$ for $i= 1,...,n$, and that the application of each matrix uses $O(\binom{n}{k})$ operations. Our aim in this section is to show that the matrix ${[B_{i+1}]}_{B_{i}}$ can be computed from ${[B_{i}]}_{B_{i-1}}$ using $O(\binom{n}{k})$ operations. This will enable to compute the Fourier transform in $O(n \binom{n}{k})$ operations.

We construct  ${[B_{i+1}]}_{B_{i}}$ in two steps. First we obtain from ${[B_{i}]}_{B_{i-1}}$ the matrix ${[J_{i+1}]}_{B_{i}}$ and in a second step  we construct ${[B_{i+1}]}_{B_{i}}$ from ${[J_{i+1}]}_{B_{i}}$ .

\subsection{First step: obtaining ${[J_{i+1}]}_{B_{i}}$ from  ${[B_{i}]}_{B_{i-1}}$}

This step is based on the formula 
$$
J_{i+1} s_i = s_i J_i + 1 
$$
where $s_i$ is the the transposition $(i (i+1))$ for $i=1,...,n-1$.
From this formula we derive
\begin{align} 
\label{recurrence1}
{[J_{i+1}]}_{B_{i}} &= (\ {[s_{i}]}_{B_{i}}\  {[J_{i}]}_{B_{i}} + I \ )\  {[s_{i}]}_{B_{i}} \\
\label{recurrence2}
{[s_{i}]}_{B_{i}}  &=  {[B_{i-1}]}_{B_{i}} \  {[s_{i}]}_{B_{i-1}} \  {[B_{i}]}_{B_{i-1}}
\end{align}

Recall that the bases $B_i$ are orthonormal, so that ${[B_{i-1}]}_{B_{i}}$ is the transpose of ${[B_{i}]}_{B_{i-1}}$. Equations (\ref{recurrence1}) and (\ref{recurrence2}) show that we can compute ${[J_{i+1}]}_{B_{i}}$ from ${[B_{i}]}_{B_{i-1}}$ provided we know the matrices $ {[s_{i}]}_{B_{i-1}}$ and ${[J_{i}]}_{B_{i}}$. On the one hand we observe that ${[s_{i}]}_{B_{i-1}}$ is a permutation matrix which is easily described in terms of the labels of the basis $B_{i-1}$. On the other hand, we can see that the matrix ${[J_{i}]}_{B_{i}}$ is a diagonal matrix whose diagonal entries are the eigenvalues described by formula (\ref{eigenvalues}).

We claim that once ${[s_{i}]}_{B_{i-1}}$,  ${[J_{i}]}_{B_{i}}$ and ${[B_{i}]}_{B_{i-1}}$ have been computed, the computation of ${[J_{i+1}]}_{B_{i}}$ using  (\ref{recurrence1}) and (\ref{recurrence2}) can be performed using $O(\binom{n}{k})$ operations. 
The key point is that all the operators and bases involved in Equations (\ref{recurrence1}) and (\ref{recurrence2}) are adapted to the decomposition
\begin{equation} \label{Key decomposition}
\mathcal{F} =  \bigoplus  \mathcal{F}_{\lambda_1...\lambda_{i-1}}^{c_{i+2} ... c_n} 
\end{equation}
that is, the decomposition in Theorem \ref{Adapted decomposition} with $p=i-1$ and $q=i+2$ (see Figure \ref{Coarser decomposition}).Since each subspace in this decomposition has dimension at most $4$, we see from Lemma \ref{Block-diagonal} that if $B_{i}$ and $B_{i-1}$ are properly ordered  then all the matrices appearing in  (\ref{recurrence1}) and (\ref{recurrence2}) are block-diagonal with each block of size at most $4$. Let us prove this fact.

\begin{figure}[ht] 
\begin{center}
\includegraphics[scale=0.8]{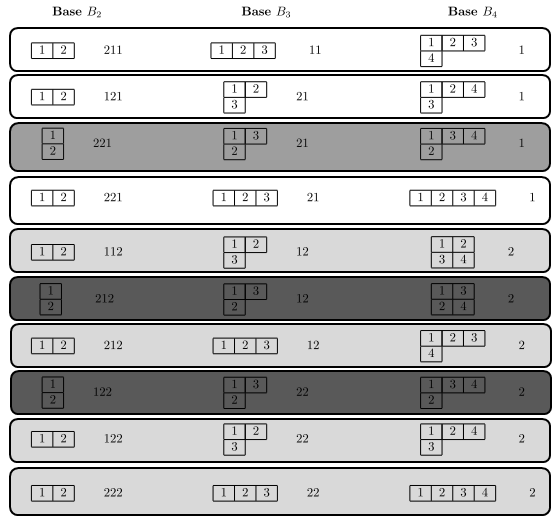} 
\caption{An illustration of the case $n=5$, $k=2$, $i=3$. The three bases $B_{i-1}$, $B_{i}$ and $B_{i+1}$ are adapted to the decomposition $\bigoplus  \mathcal{F}_{\lambda_1...\lambda_{i-1}}^{c_{i+2} ... c_n} $. Each element of these bases is coloured according to the subspace of the decomposition containing it.}
\label{Coarser decomposition}
\end{center}
\end{figure}

\begin{lemma} \label{Size 4}
For $i=1,...,n-1$ let 
$$
\mathcal{F} =  \bigoplus  \mathcal{F}_{\lambda_1...\lambda_{i-1}}^{c_{i+2} ... c_n} 
$$
be the decomposition in Theorem \ref{Adapted decomposition} with $p=i-1$ and $q=i+2$. 
Then 

a) the dimension of each subspace $ \mathcal{F}_{\lambda_1...\lambda_{i-1}}^{c_{i+2} ... c_n} $ is at most $4$

b) the operators $s_i$ and $J_i$ are adapted to this decomposition 

c) the bases $B_{i-1}$, $B_{i}$ and $B_{i+1}$ are adapted to this decomposition
 \end{lemma}
\begin{proof}
For a) observe that the subspace $ \mathcal{F}_{\lambda_1...\lambda_{i-1}}^{c_{i+2} ... c_n} $ is spanned by the four subspaces 
\begin{equation} \label{Four subspaces}
 \mathcal{F}_{\lambda_1...\lambda_{i-1}}^{11c_{i+2} ... c_n}  , \ \ 
 \mathcal{F}_{\lambda_1...\lambda_{i-1}}^{12c_{i+2} ... c_n} ,\ \ 
 \mathcal{F}_{\lambda_1...\lambda_{i-1}}^{21c_{i+2} ... c_n} ,\ \ 
 \mathcal{F}_{\lambda_1...\lambda_{i-1}}^{22c_{i+2} ... c_n} 
\end{equation}
and according to Theorem \ref{One-dimensional decomposition} these subspaces have dimension at most one.

For b) note that the operator $s_i$ acts by interchanging the subspaces in (\ref{Four subspaces}), and these subspaces span $ \mathcal{F}_{\lambda_1...\lambda_{i-1}}^{c_{i+2} ... c_n} $,  so $s_i$ stabilizes it. Similarly, to see that $J_i$ stabilizes it, observe that 
$ \mathcal{F}_{\lambda_1...\lambda_{i-1}}^{c_{i+2} ... c_n} $ is spanned by the subspaces of the form
$ \mathcal{F}_{\lambda_1...\lambda_{i-1} \lambda}^{c\ c_{i+2} ... c_n} $
where $\lambda$ is a Young diagram and $c$ is a letter. Then
by Corollary \ref{Eigenvector} the operator $J_i$ stabilizes each one of these last subspaces.

Part c) is an instance of Lemma \ref{Adapted bases}.
\end{proof}

\begin{prop} \label{Step one}
Let us assume that the matrix ${[B_{i}]}_{B_{i-1}}$ have been computed for some $i$ such that $1 \leq i \leq n-1$. Then the matrix ${[J_{i+1}]}_{B_{i}}$ can be computed in  $O(\binom{n}{k})$ operations.
\end{prop}

\begin{proof}
The algorithm for computing ${[J_{i+1}]}_{B_{i}}$ proceed by restricting the bases and the operators appearing in (\ref{recurrence1}) and (\ref{recurrence2})  to each subspace  $\mathcal{F}_{\lambda_1...\lambda_{i-1}}^{c_{i+2} ... c_n}$ in the decomposition  (\ref{Key decomposition}), one at a time.  
Once the operators and the bases have been restricted to  $\mathcal{F}_{\lambda_1...\lambda_{i-1}}^{c_{i+2} ... c_n}$, by Lemma \ref{Size 4} the matrices in (\ref{recurrence1}) and (\ref{recurrence2}) are square matrices of size at most $4$. 

The matrix ${[s_{i}]}_{B_{i-1}}$ is a permutation matrix corresponding to the permutation of the subspaces in (\ref{Four subspaces}). By  Corollary \ref{Eigenvector}, ${[J_{i}]}_{B_{i}}$  is a diagonal matrix whose diagonal entries can be computed using formula (\ref{eigenvalues}) for the eigenvalues of the Jucys-Murphy operators.

The formulas (\ref{recurrence1}) and (\ref{recurrence2}) involve four matrix multiplications and one matrix sum, so the number of arithmetic operations required to compute ${[J_{i+1}]}_{B_{i}}$ is at most $4 \   4^2 + 4 = 68$.  Since the number of subspaces in the decomposition (\ref{Key decomposition}) is at most $\binom{n}{k}$, then the total number of arithmetic operations does not exceed $68 \binom{n}{k}$.
\end{proof}

\subsection{Second step: obtaining ${[B_{i+1}]}_{B_{i}}$ from ${[J_{i+1}]}_{B_{i}}$}

\begin{prop} \label{Step two}
Let us assume that the matrix ${[J_{i+1}]}_{B_{i}}$ have been computed for some $i$ such that $1 \leq i \leq n-1$. Then the matrix ${[B_{i+1}]}_{B_{i}}$ can be computed in  $O(\binom{n}{k})$ operations.
\end{prop}
\begin{proof}
The elements of the basis $B_{i+1}$ are non-zero vectors in the subspaces of the form 
$ \mathcal{F}_{\lambda_1...\lambda_{i+1}}^{c_{i+2} ... c_n}$. By Corollary \ref{Eigenvector}, such vectors are eigenvectors of the operator $J_{i+1}$. In order to find all such eigenvectors written in the basis $B_i$ we must find the column vectors $[v]_{B_i}$ such that 
\begin{equation} \label{Column eigenvectors}
(\ [J_{i+1}]_{B_i} - \lambda I\ ) \ [v]_{B_i} = 0
\end{equation}
for some eigenvalue $\lambda$ of $J_{i+1}$. Once again, the calculation can be simplified by restricting to appropriate subspaces. In effect, the operator $J_{i+1}$ and the bases $B_{i+1}$ and $B_{i}$ are adapted to the decomposition 
\begin{equation} \label{Second step}
\mathcal{F} =  \bigoplus  \mathcal{F}_{\lambda_1...\lambda_{i}}^{c_{i+2} ... c_n} 
\end{equation}
whose subspaces have dimension at most two. We proceed by restricting to each subspace and obtaining the corresponding blocks of $[B_{i+1}]_{B_i}$, one at a time. When the operator and the bases are restricted to $\mathcal{F}_{\lambda_1...\lambda_{i}}^{c_{i+2} ... c_n}$,  equation (\ref{Column eigenvectors}) is a homogeneus linear system of equations in at most two variables.

If $\mathcal{F}_{\lambda_1...\lambda_{i}}^{c_{i+2} ... c_n}$ is one-dimensional, then the corresponding block of $[B_{i+1}]_{B_i}$ is of size one with $1$ as the only entry, so we are done.
If $\mathcal{F}_{\lambda_1...\lambda_{i}}^{c_{i+2} ... c_n}$ is two-dimensional, then there are two different eigenvalues which correspond to the two ways in which a letter can be inserted by the Robinson-Schensted insertion step. If the form of the Young diagram $\lambda_i$ is  $(i-a,a)$, then, by the eigenvalue formula (\ref{eigenvalues}), the two eigenvalues of the restriction of $J_{i+1}$  to $ \mathcal{F}_{\lambda_1...\lambda_{i}}^{c_{i+2} ... c_n}$ are $i-a$ and $a-1$ (see Figure \ref{Possible eigenvalues}).

\begin{figure}[ht] 
\begin{center}
\includegraphics[scale=0.35]{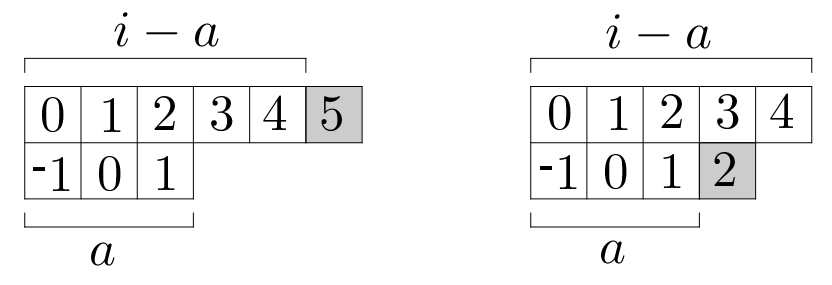} 
\caption{The number in the boxes indicates their contents. The content of the inserted box is the eigenvalue of the operator $J_{i+1}$, which can be $i-a$ (left) or $a-1$ (right).}
\label{Possible eigenvalues}
\end{center}
\end{figure}

Solving each of the two-variable linear systems
\begin{equation} \label{i-a}
(\ [J_{i+1}]_{B_i} - (i-a) I\ ) \ [v]_{B_i} = 0  
\end{equation}
\begin{equation} \label{a-1}
(\ [J_{i+1}]_{B_i} - (a-1) I\ ) \ [v]_{B_i} = 0
\end{equation}
we obtain the two eigenvectors of $J_{i+1}$ restricted to $ \mathcal{F}_{\lambda_1...\lambda_{i}}^{c_{i+2} ... c_n}$. Normalizing each eigenvector we construct an orthogonal matrix of size two which correspond to the restriction of  $[B_{i+1}]_{B_i}$ to $ \mathcal{F}_{\lambda_1...\lambda_{i}}^{c_{i+2} ... c_n}$. 

The number of arithmetic operations required to solve (\ref{i-a}) and (\ref{a-1}) is bounded by a constant that does not depend on $n$ and $k$. Since there are at most  $\binom{n}{k}$ subspaces in the decomposition (\ref{Second step}), then the number of arithmetic operations involved in the second step is $O(\binom{n}{k})$.
\end{proof}

As a consequence we obtain our main theorem.
\begin{thm} \label{Main}
The Fourier transform on the Johnson graph $J(n,k)$ can be computed using $O(n \binom{n}{k})$ arithmetic operations.
\end{thm}
\begin{proof}
The Fourier transform consists in the application of the matrix $[B_1]_{B_n}$ to a vector $[f]_{B_0}$ where 
$$
[B_0]_{B_n} =  [B_{n-1}]_{B_n}  \  \ ...\ \ [B_2]_{B_3} \ \  [B_1]_{B_2}\
$$
We start contructing $[B_1]_{B_2}$ from $ [B_0]_{B_1}=I$ using $O(\binom{n}{k})$ operations according to Propositions \ref{Step one} and \ref{Step two}, and apply it to $[f]_{B_1}$ using $O(\binom{n}{k})$ in virtue of Theorem \ref{Sparsity of matrix}, obtaining the vector  $[f]_{B_2}$. In the same way we construct $[B_2]_{B_3}$ from $ [B_1]_{B_2}$ and apply it to $[f]_{B_2}$ using $O(\binom{n}{k})$ operations to obtain $[f]_{B_3}$, and so on. Since this process finishes after $n-1$ steps, then the theorem follows.
\end{proof}

\section{Application to the computation of isotypic components}
\label{Applications}

The upper bound we obtained for the algebraic complexity of the Fourier transform can be applied to the problem of computing the isotypic projections of a given function on the Johnson graph.

For $a=0,...,s$, let $\mathcal{F}_a$ be the isotypic component of  $\mathcal{F}$ corresponding to the Young diagram $(n-a,a)$ under the action of the group $S_n$. Since these components are orthogonal and expand the space $\mathcal{F}$, given a function $f \in  \mathcal{F}$ there are uniquely determined functions $f_a \in  \mathcal{F}_a$ such that 
$$
f= \sum_{a=0}^{s} f_a
$$
For $H\subseteq \{0,...,s\}$ let $f_H$ be defined by
$$
f_H=\sum_{a \in H} f_a
$$

\begin{thm}
\label{Single isotypic component}
Assume that the matrices $[B_{i-1}]_{B_{i}}$ for $i=2,3,...,n$ have been computed. Given a column vector  $[f]_{B_0}$ with $f \in \mathcal{F}$, the column vector $[f_H]_{B_0}$ can be computed using at most $4(n-1) \binom{n}{k}$  operations.
\end{thm}
\begin{proof}
First we apply the Fourier transform to the function $f$, so that we obtain the column vector $[f]_{B_n}$ using $2( n-1) \binom{n}{k}$ operations. The basis $B_n$ is parametrized by all Young tableaux of shape $(n-a,a)$ for $a=0,...,s$. Then we substitute by $0$ the values of the entries of the vector  $[f]_{B_n}$ that correspond to Young tableaux of shape $(n-a,a)$ with $a$ not in $H$. The resulting column vector is $[f_H]_{B_n}$. Finally we apply the inverse Fourier transform to $[f_H]_{B_n}$ so that we obtain $[f_H]_{B_0}$ using $2(n-1) \binom{n}{k}$ more operations.
\end{proof}

\begin{thm}
\label{Weights of the decomposition}
Assume that the matrices $[B_{i-1}]_{B_{i}}$ for $i=2,3,...,n$ have been computed. Given a column vector  $[f]_{B_0}$ with $f \in \mathcal{F}$, all the weights $\|f_a\|^2$, for $a=0,...,s$,   can be computed using at most $(2n-1) \binom{n}{k}$ operations.
\end{thm}
\begin{proof}
Observe that $\|f_a\|^2 = \|[f_a]_{B_n}\|^2 $. To obtain the  column vector $[f_a]_{B_n}$, we apply the Fourier transform to the function $f$, so that we obtain the column vector $[f]_{B_n}$ using $2( n-1) \binom{n}{k}$ operations. Then we select the entries of the vector  $[f]_{B_n}$ that correspond to Young tableaux of shape $(n-a,a)$, and we compute the sum of the squares of these entries. Doing this for all the values of $a$ can be accomplished using at most $\binom{n}{k}$ operations.
\end{proof}

\end{document}